\documentclass[12pt]{amsart}
\usepackage[utf8]{inputenc}
\usepackage{amssymb}
\usepackage{amsmath}
\usepackage{amsfonts}
\usepackage{amsthm}
\usepackage{tikz}
\usepackage{algorithmic}
\usepackage{asymptote}
\usepackage{url}
\newtheorem{theorem}{Theorem}[section]
\newtheorem{corollary}{Corollary}[theorem]
\newtheorem{lemma}[theorem]{Lemma}

\newtheorem{conjecture}[theorem]{Conjecture}
\newtheorem{definition}[theorem]{Definition}

\theoremstyle{definition}

\newtheorem{example}[theorem]{Example}

\begin{document}
\title{Distribution of genus among numerical semigroups with fixed Frobenius number}
\author{Deepesh Singhal}
\address{Hong Kong University of Science and Technology}
\email{dsinghal@connect.ust.hk}

\begin{abstract}
    A numerical semigroup is a sub-semigroup of the natural numbers that has a finite complement. The size of its complement is called the genus and the largest number in the complement is called its Frobenius number. We consider the set of numerical semigroups with a fixed Frobenius number $f$ and analyse their genus. We find the asymptotic distribution of genus in this set of numerical semigroups and show that it is a product of a Gaussian and a power series. We show that almost all numerical semigroups with Frobenius number $f$ have genus close to $\frac{3f}{4}$.
    We denote the number of numerical semigroups of Frobenius number $f$ by $N(f)$. While $N(f)$ is not monotonic we prove that $N(f)<N(f+2)$ for every $f$.
\end{abstract}

\maketitle

\section{Introduction}
A numerical semigroup is a subset of natural numbers that contains $0$, is closed under addition and has a finite complement with respect to the natural numbers.
The numbers that are in the complement of a numerical semigroup $S$ are called its gaps.
The number of gaps is called the genus, it is denoted by $g(S)$.
The largest gap is called the Frobenius number, it is denoted by $f(S)$.
The smallest non-zero element of $S$ is called its multiplicity and is denoted by $m(S)$.

For a given $f$ there are only finitely many numerical semigroups with Frobenius number $f$ (at most $2^{f-1}$).
Denote the number by $N(f)$.
Backelin in \cite{Backelin} proves the following theorem, here $\overline{f}=\left\lfloor \frac{f-1}{2}\right\rfloor$.
\begin{theorem}\label{NS with f}\cite[Proposition 1]{Backelin}
The following limits exist and are positive (we denote the values by $c_1$, $c_2$)
$$\lim_{f\text{ odd}} \frac{N(f)}{2^{\overline{f}}} =c_1$$
$$\lim_{f\text{ even}} \frac{N(f)}{2^{\overline{f}}} =c_2.$$
\end{theorem}

In \cite{N.S. with Frob}, the authors study the the set of numerical semigroups with Frobenius number $f$ and give an algorithm to compute it.
In \cite{N.S. with Frob and Mult}, the authors do the same for the set of numerical semigroups with a given Frobenius number and multiplicity.
Both these papers use a similar strategy of partitioning the respective sets into equivalence classes such that numerical semigroups belong to the same class if they have the same
$$X(S)=\left\{x\in S\mid 1\leq x< \frac{f}{2}\right\}.$$

In this paper we will study how these equivalence classes change when we vary $f$.
We will thus give a detailed description of numerical semigroups for which $f(S)<3m(S)$ and it will be shown that almost all numerical semigroups satisfy this property.
We will use this to analyse the distribution of genus among numerical semigroups with Frobenius number $f$.
In Section \ref{average section} we show that
\begin{theorem}\label{average of genus}
Among numerical semigroups with Frobenius number $f$ the average value of genus is $\frac{3f}{4}+o(f)$.
\end{theorem}
In fact most numerical semigroups have genus close to this value, in Section \ref{average section} we also show that
\begin{theorem}\label{most genus near 3f/4}
For any $\epsilon>0$,
$$\lim_{f\to\infty} \frac{\#\{S\mid f(S)=f, |g(S)-\frac{3f}{4}|<f^{\frac{1}{2}+\epsilon}\}}{N(f)}=1.$$
\end{theorem}

We obtain the limiting distribution of genus among numerical semigroups with Frobenius number $f$ in Theorem \ref{Main Distribution}.
It is of the form of a Gaussian times a power series.

Let $n_g$ be the number of numerical semigroups of genus $g$.
Bras-Amorós conjectured in \cite{Maria} that the sequence $n_g$ is monotonic.
However when we count by Frobenius number, it is known that $N(f)$ is not monotonic for example $N(5)=5, N(6)=4$.
However since $N(f)$ behaves differently for even and odd $f$, it makes sense to investigate whether the two sub sequences for even and odd $f$ are monotonic. Indeed they must be eventually monotonic by Theorem \ref{NS with f}. And we prove in Section \ref{monotone section} that
\begin{theorem}\label{N(f) is monotonic}
For every positive integer $f$, $N(f)<N(f+2)$.
\end{theorem}

It is well known that each numerical semigroup has a unique minimal set of generators, the number of generators is called the embedding dimension and is denoted by $e(S)$.
It is known that $e(S)\leq m(S)$, if equality holds for a numerical semigroup $S$ then it is said to be of max embedding dimension. Let $MED(f)$ be the number of max embedding dimension numerical semigroups with Frobenius number $f$.
In Theorem \ref{bounds for MED by Frob} we prove that
\begin{theorem}
There are positive constants $c$, $c'$ such that for any $f$
$$c2^{\frac{1}{3} f}<MED(f)<c'2^{0.41385f}.$$
\end{theorem}
We also make a conjecture about the growth of $MED(f)$.
\begin{conjecture}
The following limit exists
$$\lim_{f\to\infty}\frac{\text{log}_{2}(MED(f))}{f}.$$
\end{conjecture}

We can also look at the set of numerical semigroups of a fixed genus $g$. 
The following is proved in \cite{Zhai}, here $\phi$ is the golden ratio.
\begin{theorem}\cite[Theorem 1]{Zhai}
The following limit exists 
$$\lim_{g\to\infty}\frac{n_g}{\phi^{g}}.$$
\end{theorem}

Nathan and Ye in \cite{Nathan} studied the set of numerical semigroups with a fixed genus. They proved that almost all of them have Frobenius number close to twice the multiplicity.
\begin{theorem}\cite{Nathan}
Let $\epsilon>0$ then
$$\lim_{g\to\infty}\frac{\#\{S|g(S)=g, |F(S)-2m(S)|>\epsilon g\}} {n_g}=0.$$
\end{theorem}

We make this result stronger and prove in Section \ref{genus} that
\begin{theorem}\label{F 2m genus result}
For any $\epsilon>0$ there is an $N$ such that for sufficiently large $g$
$$\frac{\#\{S|g(S)=g, |F(S)-2m(S)|>N\}}{n_g}<\epsilon.$$
\end{theorem}

\section{Depth of a numerical semigroup}
Notation: Throughout this paper, $\overline{f}= \left\lfloor \frac{f-1}{2}\right\rfloor$.
Also we will use intervals $[a,b]=\{n\in\mathbb{Z}\mid a\leq n\leq b\}$,
$(a,b)=\{n\in\mathbb{Z}\mid a< n< b\}$.

The depth of a numerical semigroup $S$ is $q(S)=\left\lceil \frac{f(S)+1}{m(S)}\right\rceil$.
In particular, $q=1$ only for the numerical semigroup $\{0,f+1\rightarrow\}$. Here, the arrow indicates that all natural numbers after $f+1$ are in the semigroup.
Moreover, $q=2$ when $\frac{f}{2}<m<f$; $q=3$ when $\frac{f}{3}<m\leq \overline{f}$ and $q\geq 4$ when $m<\frac{f}{3}$. Most numerical semigroups have depth $2$ or $3$ (Corollary \ref{depth >3}), we will primarily be interested in these two families.

\begin{theorem}\label{m near f/2}
For any $\epsilon>0$, there is a $M$ such that for all $f$
$$\frac{\#\{S\mid f(S)=f,|m(S)-\frac{f}{2}|>M\}} {2^{\overline{f}}} <\epsilon.$$
\end{theorem}
\begin{proof}
See Proposition 2 of \cite{Backelin}.
\end{proof}
\begin{corollary}\label{depth >3}
$$\lim_{f\to\infty} \frac{\#\{S\mid f(S)=f, q(S)>3\}} {2^{\overline{f}}} =0.$$
\end{corollary}
\begin{proof}
Note that $q(S)\geq 4$ implies $|m(S)-\frac{f(S)}{2}|>\frac{f(S)}{6}$.
\end{proof}

We denote by $n(S)$ the size of $S\cap[1,f(S)]$. Therefore, $n(S)+g(S)=f(S)$. Note that for any $x$ in $[1,f-1]$, at least one of $x,f-x$ must be a gap of $S$ and hence $g(S)\geq \left\lceil \frac{f+1}{2} \right\rceil$ i.e. $n(S)\leq \overline{f}$. Among numerical semigroups with Frobenius number $f$, studying the distribution of $g(S)$ is equivalent to studying the distribution of $n(S)$. We shall be using $n(S)$ as it makes the expressions simpler.
Denote by $N(f,n)$ the number of numerical semigroups with Frobenius number $f$ and $n(S)=n$.

We now enumerate the numerical semigroups of depth $2$.
\begin{theorem}\label{depth 2 count}
There are $2^{\overline{f}}-1$ numerical semigroups of depth $2$ and Frobenius number $f$.
Moreover, there are ${\overline{f}}\choose{n}$ numerical semigroups of depth $2$, Frobenius number $f$ and $n(S)=n$.
\end{theorem}
\begin{proof}
The first part follows from the second as we add up ${\overline{f}}\choose{n}$ for $1\leq n\leq \overline{f}$. For the second part fix a $n$, pick a subset $T$ of $[f-\overline{f},f-1]$ of size $n$, this can be done in ${\overline{f}}\choose{n}$ ways. Once we have such a $T$, the sum of any two numbers in $T$ is larger than $f$. Therefore $\{0\}\cup T\cup\{f+1\rightarrow\}$ is a numerical semigroup. It is clear that all numerical semigroups of depth $2$ are achieved this way.
\end{proof}

\section{Partition of the set of numerical semigroups}\label{Partition}

In this section we will  describe a partition of the collection of all numerical semigroups. We will then study the equivalence classes of the partition in the next section.

\begin{definition}
Given a numerical semigroup $S$ with Frobenius number $f$, let
$Y(S)=\{t\mid \overline{f} -t\in S,\; 0\leq t<\overline{f}\}$.
\end{definition}
We define an equivalence relation on the set of all numerical semigroups.
Numerical semigroups $S$ and $S'$ are related if $Y(S)=Y(S')$.
If $S$ and $S'$ have the same Frobenius number then $Y(S)=Y(S')$ if and only if $X(S)=X(S')$.
Therefore, if we restrict ourselves to numerical semigroup of a fixed Frobenius number then we get the partition introduced in \cite{N.S. with Frob} and if we restrict to numerical semigroups with a fixed Frobenius number and multiplicity then we get the partition introduced in \cite{N.S. with Frob and Mult}.

\begin{definition}
Given a finite subset $Y\subseteq \mathbb{Z}_{+}$, denote by $N(Y,f)$ the number of numerical semigroups with Frobenius number $f$ and $Y(S)=Y$.
\end{definition}

It is clear that $Y(S)=\emptyset$ if and only if $S$ has depth $1$ or $2$.
Therefore $N(\emptyset,f)=1+2^{\overline{f}}-1=2^{\overline{f}}$.
Next, if $S$ has depth at least $3$ i.e. if $Y(S)\neq\emptyset$ then we will denote by $l(S)$ the largest number in $Y(S)$.
The multiplicity of $S$ is $\overline{f(S)}-l(S)$.
Therefore, $S$ has depth $3$ when $f(S)>6l(S)+6$.
Now, we restate theorem \ref{m near f/2} as follows.

\begin{theorem}\label{In terms of Y}
For every $\epsilon>0$ there exists a $L$ such that for every $f$
$$\frac{N(\emptyset,f)+\sum_{Y:Max(Y)\leq L} N(Y,f)}{N(f)}>1-\epsilon.$$
\end{theorem}

In order to study the distribution $n(S)$, we pick a large $L$, restrict ourselves to to numerical semigroups that have $Y(S)=\emptyset$ or $l(S)\leq L$.
For $f>6L+6$, all such numerical semigroups will be of depth $\leq 3$.
We will study the limit of the distribution of $n(S)$ among these semigroups as $f$ goes to infinity.
We already know how $n(S)$ behaves among numerical semigroups of depth $2$, in the next section we will study the ones with depth $3$.

\section{Numerical semigroups of depth 3}

In this section we will describe numerical semigroups of depth $3$. For $Y\neq\emptyset$ we will study the equivalence class of numerical semigroups with $Y(S)=Y$. However we need to further partition these equivalence classes first.

\begin{definition}
For a numerical semigroup $S$ of depth $3$ we define
$$Z(S)= \left\{x-f(S)+\overline{f(S)} \mid x\in S, \frac{f(S)}{2} <x\leq f(S)-m(S) \right\}.$$
\end{definition}

Note that all numbers in $Z(S)$ are non-negative and the largest number in $Z(S)$ is at most
$$\overline{f(S)}- m(S)
=\overline{f(S)}-(\overline{f(S)}-l(S))
=l(S).$$
Moreover $Y(S)\cap Z(S)=\emptyset$, this because if $x\in Y(S)\cap Z(S)$ then $\overline{f(S)}-x\in S$ and $x+f(S)-\overline{f(S)}\in S$.
Adding the two leads to $f(S)\in S$, which is impossible.

\begin{definition}
If $Y$ is a finite non-empty subset of natural numbers, $f>6Max(Y)+6$ and $Z$ is a subset $[0,Max(Y)]$ such that $Z\cap Y=\emptyset$ then we define
$N(Y,Z,f)$ to be the number of numerical semigroups $S$ with Frobenius number $f$, $Y(S)=Y$ and $Z(S)=Z$.
Also let $N(Y,Z,f,n)$ be the number of numerical semigroups with additional condition that $n(S)=n$.
\end{definition}

\begin{example}
In this section we will classify the numerical semigroups with a given $Y$, $Z$ and $f>6Max(Y)+6$.
Let us first consider an example, say $Y=\{2\}$, $Z=\{0\}$ and $f=30$ (so $\overline{f}=14$).
$Y$ tells us that $S\cap [1,14]=\{12\}$, $Z$ tells that $S\cap [16,18]=\{16\}$.
This still leaves the elements in the interval $[19,29]$ to be decided.
However, note that $12\in S$ forces $12+12=24\in S$, we will show that such forced elements correspond to the set $2Y=\{y_1+y_2\mid y_1,y_2\in Y\}$.
Another forced element is $12+16=28$, we will show that such forced elements correspond to $W_1(Y,Z)$ or $W_2(Y,Z)$ (defined below) depending on the parity of $f$.
It can be seen that the remaining numbers in $[19,29]\setminus\{24,28\}$ can be independently included or excluded from $S$.
Therefore, there are $2^{9}$ such numerical semigroups.
\end{example}

Given sets $A,B\subseteq \mathbb{Z}$ and $n\in \mathbb{Z}$ we have the following notation $2A=\{x+y|x,y\in A\}$, $A-B=\{x-y|x\in A, y\in B\}$ and $n-A=\{n-x|x\in A\}$.

\begin{definition}
Given $Y,Z$ we define
$$W_1(Y,Z)=(Y-Z-1)\cap [0,\infty)$$
$$W_2(Y,Z)=(Y-Z-2)\cap [-1,\infty).$$
\end{definition}

\begin{lemma}\label{f necessary}
If $S$ is a numerical semigroup of depth $3$ with $Y(S)=Y$, $Z(S)=Z$, $f(S)=f$ and $l(S)=l$ then
$S\cap \left[1,\overline{f}\right] =\overline{f}-Y$ and
$$S\cap \left[f-\overline{f}, f-\overline{f} +l \right]= Z+f-\overline{f}.$$
Moreover if $f$ is odd then
$$S\cap \left[f-2l-1, f-1 \right] \supseteq(f-1)-(2Y\cup W_1(Y,Z)).$$
And if $f$ is even then
$$S\cap \left[f-2l-2, f-1 \right] \supseteq(f-2)-(2Y\cup W_2(Y,Z)).$$
\end{lemma}
\begin{proof}
The first part is just the definition of $Y$.
The second follows from the definition of $Z$ and the fact that $m=\overline{f}-l$.
We now prove the third one for odd $f$. We have $X=\overline{f}-Y$ and hence $2X=(f-1)-2Y$. Now, $X\subseteq S$ implies $2X\subseteq S$, also $(f-1)-2Y\subseteq [f-1-2l,f-1]$. Finally, a general element of $W_1(Y,Z)$ is of the form $y-z-1$ for $y\in Y$ and $z\in Z$, they must satisfy $y\geq z+1$. Now, $(f-1)-(y-z-1)=(z+f-\overline{f})+(\overline{f}-y)\in S$. Also $0\leq y-z-1\leq l-1$, therefore $(f-1)-(y-z-1)\subseteq [f-l,f-1]$.

For even $f$ we have $X=\overline{f}-Y$, which implies $2X=(f-2)-2Y$, also $(f-2)-2Y\subseteq [f-2-2l,f-2]$.
Next, a general element of $W_2(Y,Z)$ is of the form $y-z-2$ for $y\in Y$ and $z\in Z$, they must satisfy $y\geq z+1$.
Now, $(f-2)-(y-z-2)=(z+f-\overline{f})+(\overline{f}-y)\in S$.
Also $-1\leq y-z-2\leq l-2$, therefore $(f-2)-(y-z-2)\subseteq [f-l,f-1]$.
\end{proof}

\begin{lemma}\label{f odd sufficient}
Given a finite non-empty $Y\subseteq\mathbb{N}$, an odd $f$ such that $f>6Max(Y)+6$ and $Z$ which is a subset of $[0,Max(Y)]$ such that $Z\cap Y=\emptyset$.

For any $T$ which is a subset of
$$[f-\overline{f}+Max(Y)+1, f-1]\setminus \Big( (f-1)-(2Y\cup W_1(Y,Z))\Big).$$
Construct $S$ as
$$S=\{0\}\cup (\overline{f}-Y)\cup (Z+f-\overline{f})\cup T \cup \big((f-1)-(2Y\cup W_1(Y,Z))\big) \cup \{f+1\rightarrow\}.$$
Then $S$ is a numerical semigroup.
\end{lemma}
\begin{proof}
We need to prove that $S$ is closed under addition.
Consider $x,y\in S$, assume that $x+y\leq f$ and $x,y\neq 0$ because otherwise we have nothing to prove.
At least one of $x,y$ must be less than $\frac{f}{2}$, say $x$, then $x\in \overline{f}-Y$.

Case 1: $y<\frac{f}{2}$ then then $y$ is in $\overline{f}-Y$ as well.
Therefore, $x+y\in 2(\overline{f}-Y)= (f-1)-2Y\subseteq S$.

Case 2: $y>\frac{f}{2}$. Then
$$y\leq f-x\leq f-(\overline{f}-Max(Y))= Max(Y)+f-\overline{f}.$$
This means that $y\in Z+f-\overline{f}$.
Now, $y-f+\overline{f}\in Z$, $\overline{f}-x\in Y$ and $(\overline{f}-x)-(y-f+\overline{f})-1=f-1-x-y\in (Y-Z-1)$.
We know that $x+y\leq f$. If $x+y=f$ then $(\overline{f}-x)=(y-f+\overline{f})$ which would contradict the fact that $Y\cap Z=\emptyset$. Therefore $x+y\leq f-1$, which means $f-1-x-y\in W_1(Y,Z)$. It also implies that $x+y=(f-1)-(f-1-x-y)\in S$.
\end{proof}

\begin{lemma}\label{f even sufficient}
Given a finite non-empty $Y\subseteq\mathbb{N}$, an even $f$ such that $f>6Max(Y)+6$ and $Z$ which is a subset of $[0,Max(Y)]$ such that $Z\cap Y=\emptyset$.

For any $T$ which is a subset of
$$[f-\overline{f}+Max(Y)+1, f-1]\setminus \Big( (f-2)-(2Y\cup W_2(Y,Z))\Big).$$
Construct $S$ as
$$S=\{0\}\cup (\overline{f}-Y)\cup (Z+f-\overline{f})\cup T \cup \big((f-2)-(2Y\cup W_2(Y,Z))\big) \cup \{f+1\rightarrow\}.$$
Then $S$ is a numerical semigroup.
\end{lemma}
\begin{proof}
We need to prove that $S$ is closed under addition.
Consider $x,y\in S$, assume that $x+y\leq f$ and $x,y\neq 0$ because otherwise we have nothing to prove.
At least one of $x,y$ must be less than $\frac{f}{2}$, say $x$, then $x\in \overline{f}-Y$.

Case 1: $y<\frac{f}{2}$ then then $y$ is in $\overline{f}-Y$ as well.
Therefore, $x+y\in 2(\overline{f}-Y)= (f-2)-2Y\subseteq S$.

Case 2: $y>\frac{f}{2}$. Then
$$y\leq f-x\leq f-(\overline{f}-Max(Y))= Max(Y)+f-\overline{f}.$$
This means that $y\in Z+f-\overline{f}$.
Now, $y-f+\overline{f}\in Z$, $\overline{f}-x\in Y$ and $(\overline{f}-x)-(y-f+\overline{f})-2=f-2-x-y\in (Y-Z-2)$.
We know that $x+y\leq f$. If $x+y=f$ then $(\overline{f}-x)=(y-f+\overline{f})$ which would contradict the fact that $Y\cap Z=\emptyset$. Therefore $x+y\leq f-1$, which means $f-2-x-y\in W_1(Y,Z)$. It also implies that $x+y=(f-2)-(f-2-x-y)\in S$.
\end{proof}

We have characterised the depth $3$ numerical semigroups with a given $f,Y,Z$ we are now going to count them. 
We make the following notations $|2Y\cup W_1(Y,Z)|=\alpha$, $|2Y\cup W_2(Y,Z)|=\alpha'$, $Max(Y)+1-|Y\cup Z|=\beta$. These are all functions of $Y,Z$ of course.

\begin{theorem}\label{depth 3 count}
If $Y$ is a finite, non-empty subset of natural numbers, $f>6Max(Y)+6$ and $Z$ is a subset $[0,Max(Y)]$ such that $Z\cap Y=\emptyset$ and $f$ is odd then
$$N(Y,Z,f)=2^{\overline{f}-Max(Y)-1-\alpha}$$
$$N(Y,Z,f,n)={\overline{f}-Max(Y)-1-\alpha\choose n-Max(Y)-1-\alpha+\beta}.$$
If $f$ is even the replace $\alpha$ with $\alpha'$.
\end{theorem}
\begin{proof}
This follows from Lemma \ref{f necessary}, Lemma \ref{f odd sufficient} and Lemma \ref{f even sufficient}.
\end{proof}

\section{Monotonicity of $N(f)$}\label{monotone section}


Denote by $N_{mul}(m,f)$ the number of numerical semigroups with Frobenius number $f$ and multiplicity $m$.
We know that
$N(\emptyset,f+2)-N(\emptyset,f)=2^{\overline{f}}$,
and for $\frac{f}{3}<m<\frac{f}{2}$ $N_{mul}(m,f)<N_{mul}(m+1,f+2)$ as
$$N_{mul}(m,f)
=\sum_{Max(Y)=\overline{f}-m} \sum_{Z} 2^{m-1-\alpha}
<\sum_{Max(Y)=\overline{(f+2)}-(m+1)} \sum_{Z} 2^{m-\alpha}$$
$$=N_{mul}(m+1,f+2).$$
Here we use Theorem \ref{depth 3 count}, replace $\alpha$ with $\alpha'$ if $f$ is even.

\begin{lemma}
For $m<\frac{f}{2}$
$$N_{mul}(m,f)\leq \frac{1}{4}2^{\overline{f}}\left(\frac{11}{12}\right)^{\overline{f}-m}.$$
\end{lemma}
\begin{proof}
See \cite{Backelin}.
\end{proof}

\begin{corollary}
$$\sum_{m<\frac{f}{3}}N_{mul}(m,f)\leq 3\left(\frac{11}{12}\right)^{\frac{f}{6}-\frac{1}{2}}2^{\overline{f}}.$$
\end{corollary}

\begin{proof}[Proof of Theorem \ref{N(f) is monotonic}]
In \cite{OEIS} the value of $N(f)$ are listed for $f\leq 39$, this can thus be checked manually for $f\leq 37$. We now assume $f>37$.
We have computed
$$\sum_{Y:Max(Y)\leq 5}\sum_{Z\subseteq [0,Max(Y)-1],Z\cap Y=\emptyset} 2^{-Max(Y)-1-\alpha(Y,Z)}>1.08$$
$$\sum_{Y:Max(Y)\leq 5}\sum_{Z\subseteq [0,Max(Y)-1],Z\cap Y=\emptyset} 2^{-Max(Y)-1-\alpha'(Y,Z)}>1.06.$$
It follows that
$$\sum_{m>\frac{f}{3}}N(m+1,f+2)-N(m,f)
> 2^{\overline{f}}(1+1.06).$$
Moreover
$$\sum_{m<\frac{f}{3}}N_{mul}(m,f)\leq 3\left(\frac{11}{12}\right)^{\frac{37}{6}-\frac{1}{2}}2^{\overline{f}}
<1.9\times 2^{\overline{f}}.$$
It follows that $N(f)<N(f+2)$.
\end{proof}

\section{Expectation of genus given Frobenius number}\label{average section}

In this section we will use Theorem \ref{depth 2 count} and Theorem \ref{depth 3 count} to find the expected value of genus among numerical semigroups of fixed Frobenius number.
We will thus prove Theorem \ref{average of genus} and Theorem \ref{most genus near 3f/4}.
Before that we give expressions for the constants $c_1,c_2$ from Theorem \ref{NS with f}.

\begin{theorem}
The constants $c_1,c_2$ of Theorem \ref{NS with f} are given by
$$c_1=1+\sum_{Y\neq\emptyset, |Y|<\infty} \sum_{Z\subseteq [0,Max(Y)]\setminus Y} 2^{-Max(Y)-1-\alpha},$$
$$c_2=1+\sum_{Y\neq\emptyset, |Y|<\infty} \sum_{Z\subseteq [0,Max(Y)]\setminus Y} 2^{-Max(Y)-1-\alpha'}.$$
\end{theorem}
\begin{proof}
Let $\epsilon>0$, consider the $L$ given by Theorem \ref{In terms of Y}. So for each $f$
$$(1-\epsilon) \frac{N(f)}{2^{\overline{f}}}
\leq \frac{N(\emptyset,f)
+\sum_{Y:Max(Y)\leq L} N(Y,f)} {2^{\overline{f}}}
\leq \frac{N(f)}{2^{\overline{f}}}.$$
Now for odd $f>6L+6$
$$\frac{N(\emptyset,f)+\sum_{Y:Max(Y)\leq L} N(Y,f)} {2^{\overline{f}}}
=1+\sum_{Y\neq\emptyset, Max(Y)\leq L} \sum_{Z\subseteq [0,Max(Y)]\setminus Y} 2^{-Max(Y)-1-\alpha}.$$
By letting $f$ tend to infinity we get
$$(1-\epsilon)c_1 \leq
1+\sum_{Y\neq\emptyset, Max(Y)\leq L} \sum_{Z\subseteq [0,Max(Y)]\setminus Y} 2^{-Max(Y)-1-\alpha} 
\leq c_1.$$
The second inequality is true for any $L$, so the sum
$$1+\sum_{Y\neq\emptyset, Max(Y)<\infty} \sum_{Z\subseteq [0,Max(Y)]\setminus Y} 2^{-Max(Y)-1-\alpha}.$$
converges and the value is at most $c_1$. On the other hand we can pick $\epsilon$ to be arbitrarily small, so the sum is exactly $c_1$.
The equation for $c_2$ is obtained similarly by considering even $f$.
\end{proof}

We now compute the average value of genus, remember that $g(S)=f(S)-n(S)$.

\begin{proof}[Proof of Theorem \ref{average of genus}]
Pick $\epsilon>0$, consider the $L$ given by Theorem \ref{In terms of Y}.
Suppose $f>6L+6$. Consider a $Y$ which is a non-empty subset of $[0,L]$, and $Z$ a subset $[0,Max(Y)]\setminus Y$.
If $f$ is odd then the average value of $n(S)$ among numerical semigroups with $f(S)=f$, $Y(S)=Y$, $Z(S)=Z$ is
$$\frac{\overline{f}}{2}+\frac{Max(Y)+1+\alpha}{2}-\beta
=\frac{f}{4}+O_L(1).$$
This was obtained from Theorem $\ref{depth 3 count}$, if $f$ is even then replace $\alpha$ with $\alpha'$.
And by Theorem \ref{depth 2 count} the average of $n(S)$ among numerical semigroups with $f(S)=f$, $Y(S)=\emptyset$ is $\frac{\overline{f}}{2}$.

It follows that the difference between the average of $n(S)$ and $\frac{f}{4}$ is at most $2\epsilon f+O_L(1)$. Since $\epsilon$ was arbitrary we get that the average value of $n(S)$ among numerical semigroups with Frobenius number $f$ is $\frac{f}{4}+o(f)$.
This means the average value of genus is $\frac{3f}{4}+o(f)$.
\end{proof}

Next we will show that for almost all numerical semigroups the genus is close to this value.

\begin{proof}[Proof of Theorem \ref{most genus near 3f/4}]
We need the following property of binomial coefficients: for a fixed large $M$, the distribution $\frac{1}{2^M} {M\choose n}$ is approximately the Gaussian distribution with mean $\frac{M}{2}$ and standard deviation $\sigma_M=\frac{\sqrt{M}}{2}$ by the De Moivre–Laplace theorem. And since $\frac{M^{\frac{1}{2}+\epsilon}}{\sigma_M}$ goes to infinity as $M$ goes to infinity we get
$$\lim_{M\to\infty}\frac{1}{2^M}\sum_{n:|n-\frac{M}{2}|<M^{\frac{1}{2}+\epsilon}} {M\choose n} =1.$$
Now for fixed $Y,Z$ we get by Theorem \ref{depth 3 count} that for any $\epsilon>0$
$$\lim_{f\to\infty} \frac{\#\{S\mid Y(S)=Y, Z(S)=Z, f(S)=f, |n(S)-\frac{f}{4}| < f^{\frac{1}{2}+\epsilon}\}} {N(Y,Z,f)} =1.$$
Also by Theorem \ref{depth 2 count}, for any $\epsilon>0$
$$\lim_{f\to\infty} \frac{\#\{S\mid Y(S)=\emptyset,f(S)=f, |n(S)-\frac{f}{4}| < f^{\frac{1}{2}+\epsilon}\}} {N(\emptyset,f)} =1.$$
Therefore by Theorem \ref{In terms of Y} we conclude that for any $\epsilon>0$
$$\lim_{f\to\infty} \frac{\#\{S\mid f(S)=f, |n(S)-\frac{f}{4}| < f^{\frac{1}{2}+\epsilon}\}} {N(f)} =1.$$
Of course this is equivalent to saying that for any $\epsilon>0$
$$\lim_{f\to\infty} \frac{\#\{S\mid f(S)=f, |g(S)-\frac{3f}{4}| < f^{\frac{1}{2}+\epsilon}\}} {N(f)} =1.$$
\end{proof}

\section{Distribution of genus}
We will now obtain the distribution of the genus among numerical semigroups with a fixed Frobenius number.
We will be using the notation of falling factorials, $[n]_k=(n)(n-1)\dots (n-k+1)$.
Also remember that $g(S)=f(S)-n(S)$.

\begin{theorem}\label{Main Distribution}
Let $\psi_{f}$ be the Gaussian density function with mean $\frac{\overline{f}}{2}$ and variance $\frac{\overline{f}}{4}$. Let $c_1$ be the constant from Theorem \ref{NS with f}. Then for any $\epsilon>0$ there is a $L$ such that for sufficiently large, odd $f$ (and arbitrary $n$) the difference between $\frac{N(f,n)}{N(f)}$ and 
$$\frac{1}{c_1}\psi_f(n)\left(1+\sum_{Y\neq\emptyset, Max(Y)\leq L}\sum_{Z\subseteq [0,Max(Y)]\setminus Y} \left(1-\frac{n}{\overline{f}}\right)^{\beta}\left(\frac{n}{\overline{f}} \right)^{Max(Y)+1+\alpha-\beta} \right) $$
is less than $\epsilon$.

Replace $c_1$ with $c_2$ and $\alpha$ with $\alpha'$ to get the corresponding result for even $f$.
\end{theorem}
\begin{proof}
Given $\epsilon>0$ consider the $L$ given by Theorem \ref{In terms of Y}. Also consider $Y,Z$ with $Max(Y)\leq L$. We have
$${\overline{f}-l-1-\alpha\choose n-l-1-\alpha+\beta}
= \frac{[n]_{l+1+\alpha-\beta}}{\left[\overline{f}-\beta\right]_{l+1+\alpha-\beta}}
{\overline{f}-\beta\choose n} $$
$$= \frac{[n]_{l+1+\alpha-\beta}}{\left[\overline{f}-\beta\right]_{l+1+\alpha-\beta}}
\frac{[\overline{f}-n]_{\beta}}{[\overline{f}]_{\beta}}
{\overline{f}\choose n} 
=\frac{\prod_{i=0}^{\beta-1} (\overline{f}-n-i) \prod_{i=0}^{l+\alpha-\beta} (n-i)} {\prod_{i=0}^{l+\alpha} (\overline{f}-i)}
{\overline{f}\choose n}$$
$$=\frac{\prod_{i=0}^{\beta-1} (1-\frac{n}{\overline{f}}-\frac{i}{\overline{f}}) \prod_{i=0}^{l+\alpha-\beta} (\frac{n}{\overline{f}}-\frac{i}{\overline{f}})} {\prod_{i=0}^{l+\alpha} (1-\frac{i}{\overline{f}})} 
{\overline{f}\choose n}.$$
Therefore once we fix $l,\alpha,\beta$ we get the following limit ($n$ is allowed to vary with $f$)
$$\lim_{f\to\infty}\frac{1}{2^{\overline{f}}} {\overline{f}-l-1-\alpha\choose n-l-1-\alpha+\beta}
- \frac{1}{2^{\overline{f}}}
{\overline{f}\choose n}
\left(1-\frac{n}{\overline{f}}\right)^{\beta} \left(\frac{n}{\overline{f}}\right)^{l+1+\alpha-\beta}
=0.$$
Denote $h_L(x)=
1+\sum_{Max(Y)\leq L}\sum_{Z\subseteq [0,Max(Y)]\setminus Y} \left(1-x\right)^{\beta}x^{Max(Y)+1+\alpha-\beta}$
It follows by Theorem \ref{depth 2 count} and Theorem \ref{depth 3 count} that
$$\limsup_{f\to\infty}\left|\frac{N(f,n)}{N(f)}
-\frac{1}{c_12^{\overline{f}}} {\overline{f}\choose n} h_L\left(\frac{n}{\overline{f}}\right)\right| 
< \epsilon.$$
Now we are done by the De Moivre–Laplace theorem which implies that the binomial distribution $\frac{1}{2^{\overline{f}}}{\overline{f}\choose n}$ is approximately the Gaussian distribution with mean $\frac{\overline{f}}{2}$ and variance $\frac{\overline{f}}{4}$ for large $f$.
\end{proof}

For Frobenius numbers 19,29 we plot the distribution given in Theorem \ref{Main Distribution} with $L=2$ along with the actual distribution of $n(S)$.
The polynomial for $L=2$ is
$$1+\sum_{Max(Y)\leq 2}\sum_{Z\subseteq [0,Max(Y)]\setminus Y} \left(1-x\right)^{\beta} x^{Max(Y)+1+\alpha-\beta}
=1+2x^2-x^3+4x^4-2x^5+2x^6,$$

\includegraphics[width=0.85\textwidth]{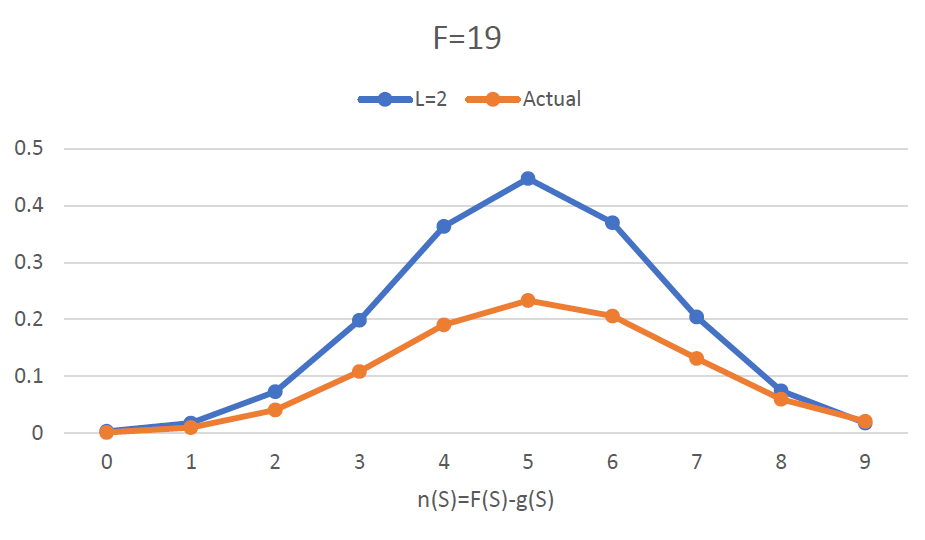}

\includegraphics[width=0.85\textwidth]{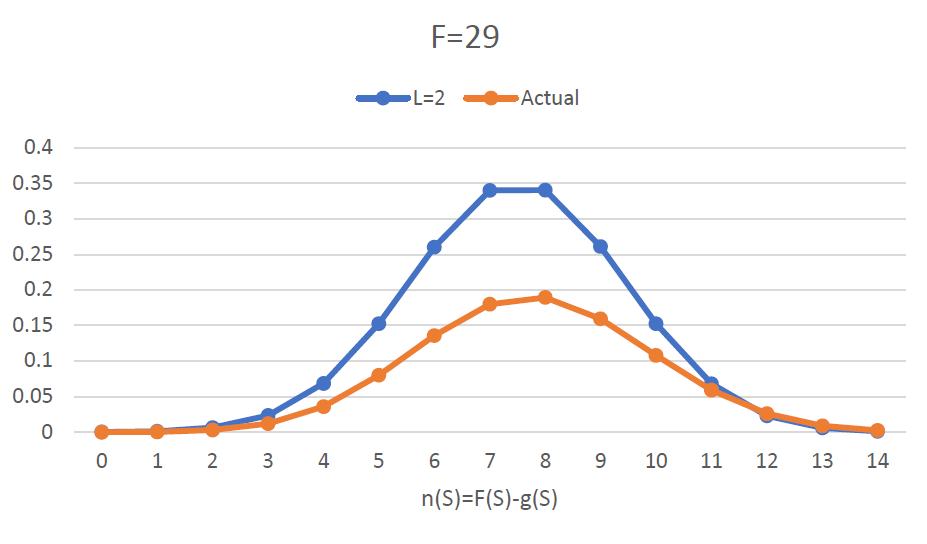}

\section{Max embedding dimension}

Let $MED(f)$ be the number of max embedding dimension numerical semigroups with Frobenius number $f$.
Let $MED(m,f)$ be the number of max embedding dimension numerical semigroups with Frobenius number $f$ and multiplicity $m$.

\begin{theorem}
Let $S$ be a numerical semigroup with multiplicity $m$. Then $S$ is of max embedding dimension if and only if $(S\setminus \{0\})-m$ is a numerical semigroup.
\end{theorem}
\begin{proof}
See \cite{NS Book}.
\end{proof}

\begin{corollary}\label{f-m}
$MED(m,f)$ is equal to the number of numerical semigroups that contain $m$ and have Frobenius number $f-m$.
\end{corollary}



To prove Theorem \ref{bounds for MED by Frob} we will need some results from \cite{Backelin}.

\begin{lemma}
For every positive integer $f$
$$2^{\overline{f}}\leq N(f)<4\times 2^{\overline{f}}.$$
\end{lemma}
\begin{proof}
See \cite{Backelin}.
\end{proof}
\begin{corollary}
For every positive integer $f$
$$\frac{1}{2} 2^{\frac{f}{2}} \leq N(f) <4\times2^{\frac{f}{2}}.$$
\end{corollary}

\begin{lemma}
Given positive integers $m,f$ such that $m<\frac{f}{4}$ we have
$$N_{mul}(m,f)< 0.071\times 2^{\frac{f}{2}}\left(\frac{13}{16}\right)^{\frac{f}{8}}2^{-0.628(\frac{f}{4}-m)}.$$
\end{lemma}
\begin{proof}
See \cite{Backelin}.
\end{proof}

\begin{theorem}\label{bounds for MED by Frob}
There are constants $c$, $c'$ such that
$$c2^{\frac{1}{3} f}<MED(f)<c'2^{0.41385f}.$$
\end{theorem}
\begin{proof}
We start with the lower bound, let $m=\left\lceil\frac{f+1}{3}\right\rceil$ so that $f=3m-r$, $1\leq r\leq 3$. Let $f_1=f-m$ so that $f_1=2m-r<2m$. Now for a lower bound on the number of numerical semigroups that contain $m$ and have Frobenius number $f_1$, just consider the depth $2$ numerical semigroups among them. Therefore by Corollary \ref{f-m},
$$MED(f)\geq MED(m,f)\geq 2^{\overline{f_1}-1}\geq 2^{\frac{f_1}{2}-2}
=\frac{1}{4}2^{\frac{f}{2}-\frac{m}{2}}
>\frac{1}{4}2^{\frac{f}{2}-\frac{f+4}{6}}
=2^{-\frac{8}{3}}2^{\frac{f}{3}}.$$

Now we obtain the upper bound, let $u=0.1723$. By Corollary \ref{f-m} we know that
$$MED(f)
=\sum_{m=2}^{f+1} MED(m,f)
\leq 1+\sum_{2\leq m<uf}N_{mul}(m,f)+\sum_{uf\leq m\leq f-1} N(f-m)$$
$$<1+0.071\times 2^{\frac{f}{2}}\left(\frac{13}{16}\right)^{\frac{f}{8}}\sum_{2\leq m<uf} 2^{-0.628(\frac{f}{4}-m)}
+4\sum_{uf\leq m\leq f-1}2^{\frac{f-m}{2}}$$
$$\ll 2^{\frac{f}{2}}\left(\frac{13}{16}\right)^{\frac{f}{8}} 2^{-0.628(0.25-u)f}+ 2^{\frac{1-u}{2}f}$$
It should be numerically checked that $\frac{1-u}{2}=0.41385$ and
$$2^{\frac{1}{2}}\left(\frac{13}{16}\right)^{\frac{1}{8}} 2^{-0.628(0.25-u)} <2^{0.41385}.$$
The result follows.
\end{proof}
In particular this means that max embedding dimension numerical semigroups have density $0$.
$$\lim_{f\to\infty}\frac{MED(f)}{N(f)}=0.$$

\begin{conjecture}
The following limit exists
$$\lim_{f\to\infty}\frac{\text{log}_{2}(MED(f))}{f}.$$
\end{conjecture}

Numerically the limit seems to be close to $0.375$. The following graph is plotted based on Table \ref{tab:MED NS}.

\includegraphics[width=\textwidth]{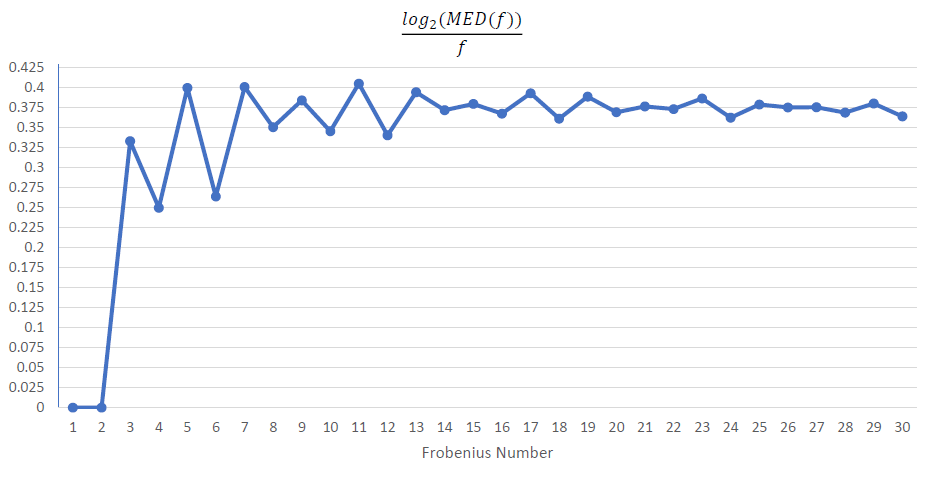}

\begin{table}[h]
    \centering
    \begin{tabular}{|c|c|c|c|c|c|c|c|c|c|}
    \hline
        $f$ & $MED(f)$& $f$ & $MED(f)$ &$f$ & $MED(f)$& $f$ & $MED(f)$ & $f$ & $MED(f)$\\
        \hline
        1 &1 &7 &7 &13 &35 &19 &168 &25 &715 \\
        2 &1 &8 &7 &14 &37 &20 &168 &26 &872 \\
        3 &2 &9 &11 &15 &52 &21 &241 &27 &1135 \\
        4 &2 &10 &11 &16 &59 &22 &298 &28 &1288 \\
        5 &4 &11 &22 &17 &103 &23 &477 &29 &2105 \\
        6 &3 &12 &17 &18 &91 &24 &418 &30 & 1949\\
        \hline
    \end{tabular}
    \caption{Number of Max ED numerical semigroups}
    \label{tab:MED NS}
\end{table}

\section{Counting by genus}\label{genus}

So far in this paper we have considered the numerical semigroups of a fixed Frobenius number, now we will consider those with a fixed genus.
In this case it is still true that most numerical semigroups have depth $2$ or $3$ as is proved in \cite{Zhai}.
\begin{theorem}\label{d 2,3}
$$\lim_{g\to\infty}\frac{\#\{S|g(S)=g, F(S)>3m(S)\}}{\phi^g}=0.$$
\end{theorem}
\begin{proof}
See \cite{Zhai}.
\end{proof}

We will therefore concentrate on numerical semigroups of depth $2,3$ as we go on to prove Theorem \ref{F 2m genus result}.
Let $F_n$ denote the $n^{th}$ Fibonacci number, it is well known that
$$F_{n+1}=\sum_{k=0}^{\left\lfloor\frac{n}{2} \right\rfloor} {n-k\choose k}.$$

\begin{theorem}\label{2m-F=k}
Fix $k\geq 1$ and $g\geq k+1$, then the number of numerical semigroups with genus $g$ and satisfying $2m(S)-F(S)=k$ is $F_{g-k}$.
\end{theorem}
\begin{proof}
If such a numerical semigroup has multiplicity $m$, then any number beyond $2m-k$ must be in $S$, $2m-k$ is not in $S$.
Let $R_0=|[m+1,2m-k-1]\cap Gap(S)|$ then $0\leq R_0\leq m-k-1$ and $g=m-1+R_0+1$ i.e. $g-(k+1)=m-k-1+R_0$.
Thus $R_0$ varies from $0$ to $\left\lfloor\frac{g-(k+1)}{2}\right\rfloor$.
Also for a particular $R_0$, the number of numerical semigroups satisfying the conditions is ${m-k-1\choose R_0}={g-(k+1)-R_0\choose R_0}$.
And hence the total number of such semigroups is
$$\sum_{R_0=0}^{\left\lfloor\frac{g-(k+1)}{2} \right\rfloor} {g-(k+1)-R_0\choose R_0}
=F_{g-k}.$$
\end{proof}

\begin{corollary}[From \cite{Nathan}, \cite{Zhao}]
The number of depth $2$ numerical semigroups with genus $g$ is $F_{g+1}-1$.
\end{corollary}
\begin{proof}
If a numerical semigroup $S$ of depth $2$ and genus $g$ then $2m(S)-F(S)\geq 1$ and $2m(S)-F(S)\leq m(S)-1\leq g-1$.
Therefore, by Theorem \ref{2m-F=k} the number of depth $2$ numerical semigroups of genus $g$ is
$$\sum_{k=1}^{g-1}F_{g-k}
=\sum_{i=1}^{g-1}F_{i}
=F_{g+1}-1.$$
\end{proof}

We can now prove one part of Theorem \ref{F 2m genus result}.
\begin{theorem}\label{part 1}
For any $\epsilon>0$ there is an $N$ such that for every $g$
$$\frac{\#\{S|g(S)=g, 2m(S)-F(S)>N\}}{\phi^g}<\epsilon.$$
\end{theorem}
\begin{proof}
By Theorem \ref{2m-F=k} it follows that
$$\frac{\#\{S|g(S)=g, 2m(S)-F(S)>N\}}{\phi^g}
=\frac{1}{\phi^g}\sum_{k=N+1}^{g-1} F_{g-k}$$
$$=\frac{1}{\phi^g}\sum_{k=1}^{g-N-1} F_{k}
<\frac{F_{g-N+1}}{\phi^g}
<\frac{2}{\sqrt{5}} \frac{\phi^{g-N+1}}{\phi^g}
=\frac{2\phi}{\sqrt{5}} \phi^{-N}.$$
\end{proof}

We now consider depth $3$ numerical semigroups. In \cite{Zhao} the following definitions are made,
$$A_k=\{A\subseteq [0,k-1]\mid 0\in A, k\not\in A+A\}.$$
Given a numerical semigroup $S$ of depth $3$, they defined the type of $S$ to be $(k;A)$, where $k=F(S)-2m(S)$ and $A=(S\cap [m,m+k])-m$.
So if $S$ has type $(k;A)$ then $A\in A_k$.

\begin{theorem}\label{F-2m=k}
If $A\in A_k$ then the number of numerical semigroups of genus $g$ and type $(k,A)$ is at most
$$F_{g-|(A+A)\cap [0,k]|+|A|-k-1}.$$
\end{theorem}
\begin{proof}
See \cite{Zhao}.
\end{proof}

\begin{theorem}\label{sum converges}
The following sum converges
$$\sum_{k}^{\infty}\sum_{A\in A_k}\phi^{-|(A+A)\cap [0,k]|+|A|-k-1}.$$
\end{theorem}
\begin{proof}
See \cite{Zhai}, \cite{Zhao}.
\end{proof}

\begin{theorem}\label{part 2}
For any $\epsilon>0$ there is an $N$ such that for every $g$
$$\frac{\#\{S|g(S)=g, m(S)>F(S)-2m(S)>N\}}{\phi^g}<\epsilon.$$
\end{theorem}
\begin{proof}
By Theorem \ref{F-2m=k} we have
$$\frac{\#\{S|g(S)=g, m(S)>F(S)-2m(S)>N\}}{\phi^g}$$
$$\leq \sum_{k=N+1}^{\infty}\sum_{A\in A_k}\frac{F_{g-|(A+A)\cap [0,k]|+|A|-k-1}}{\phi^g}
<\frac{2}{\sqrt{5}} \sum_{k=N+1}^{\infty}\sum_{A\in A_k} \phi^{-|(A+A)\cap [0,k]|+|A|-k-1}.$$
Therefore we can pick a sufficiently large $N$ by Theorem \ref{sum converges}.
\end{proof}

Theorem \ref{F 2m genus result} is therefore proved by combining Theorem \ref{part 1} and Theorem \ref{part 2}.

\section{Acknowledgements}
I would like to thank Dr Nathan Kaplan for helpful discussions about this project.


\begin{thebibliography}{}
\bibitem{Backelin}
Backelin, J. (1990).
\emph{On the number of semigroups of natural numbers.}
Mathematica Scandinavica, 66, 197-215.
https://doi.org/10.7146/math.scand.a-12304

\bibitem{N.S. with Frob}
Blanco, V.; Rosales, J.C. (2012).
\emph{On the enumeration of the set of numerical semigroups with fixed Frobenius number.}
Computers and Mathematics with Applications. 63. 1204–1211. 10.1016/j.camwa.2011.12.034. 

\bibitem{N.S. with Frob and Mult}
Branco, M.B.; Ojeda, I.; Rosales, J.C (2019).
\emph{The set of numerical semigroups of a given multiplicity and Frobenius number.}
arXiv:1904.05551v1 [math.GR]


\bibitem{Irreducible N.S. with Frob}
Blanco, V.; Rosales, J. (2013).
\emph{The tree of irreducible numerical semigroups with fixed Frobenius number.}
Forum Mathematicum, 25(6), 1249-1261.

\bibitem{NS Book}
Rosales J.C.; Garcıa-Sanchez P.A.(2010)
\emph{“Numerical Semigroups”, Developments in Maths.} 20, Springer.

\bibitem{Nathan}
Kaplan N.;Ye L. (2013).
\emph{The proportion of Weierstrass semigroups}.
Journal of Algebra,
Volume 373,
2013,
Pages 377-391,
ISSN 0021-8693,

\bibitem{Zhao}
Zhao, Y. (2010).
\emph{Constructing numerical semigroups of a given genus.}
Semigroup Forum 80, 242–254.
https://doi.org/10.1007/s00233-009-9190-9

\bibitem{Zhai}
Zhai, A. (2013).
\emph{Fibonacci-like growth of numerical semigroups of a given genus.}
Semigroup Forum 86, 634–662
https://doi.org/10.1007/s00233-012-9456-5

\bibitem{Maria}
Bras-Amorós, M. (2008).
\emph{Fibonacci-Like Behavior of the Number of Numerical Semigroups of a Given Genus.}
Semigroup Forum. 76. 379-384. 10.1007/s00233-007-9014-8. 

\bibitem{OEIS}
Garcia-Sanchez P.A.
The On-Line Encyclopedia of Integer Sequences, http://oeis.org. Sequence A124506

\end{thebibliography}
\end{document}